\documentclass[12pt]{amsart}
\usepackage{fullpage}

\usepackage{setspace}
\RequirePackage{amssymb}
\usepackage{amsbsy}
\usepackage{enumerate}
\usepackage{mathrsfs}
\usepackage{undertilde}
\usepackage{amsmath}
\usepackage{url}

\theoremstyle{plain}
\newtheorem{theorem}{Theorem}[section]
\newtheorem*{theorem*}{Theorem}
\newtheorem*{mainthm*}{Main Theorem}
\newtheorem{lemma}[theorem]{Lemma} \newtheorem*{lem*}{Lemma}
 \newtheorem*{claim*}{Claim}
 \newtheorem*{cor*}{Corollary}
\newtheorem{proposition}[theorem]{Proposition} \newtheorem*{prop*}{Proposition}

\theoremstyle{definition}
\newtheorem{definition}[theorem]{Definition} \newtheorem*{defn*}{Definition}

\theoremstyle{remark}
\newtheorem{remark}[theorem]{Remark} \newtheorem*{rem*}{Remark}
 \newtheorem*{example*}{Example}
\newtheorem{question}[theorem]{Question} \newtheorem*{question*}{Question}

\newcommand{\bfSigma}{\utilde{\bf\Sigma}}
\newcommand{\bfPi}{\utilde{\bf\Pi}}

\newcommand{\jhat}{\hat \jmath}
\DeclareMathOperator{\powerset}{\mathcal{P}}
\DeclareMathOperator{\Con}{Con}
\DeclareMathOperator{\Col}{Col}
\DeclareMathOperator{\crit}{crit}

\DeclareMathOperator{\ran}{ran}
\DeclareMathOperator{\rank}{rank}
\DeclareMathOperator{\cf}{cf}
\DeclareMathOperator{\p}{p}

\begin{document}

\title{Generic Vop\v{e}nka cardinals and models of ZF with few $\aleph_1$-Suslin sets}

\author{Trevor M.\ Wilson}

\address{Department of Mathematics\\Miami University\\Oxford, Ohio 45056\\USA}
\email{twilson@miamioh.edu} 
\urladdr{https://www.users.miamioh.edu/wilso240}

\begin{abstract}
 We define a \emph{generic Vop\v{e}nka cardinal} to be an inaccessible cardinal $\kappa$ such that for every first-order language $\mathcal{L}$ of cardinality less than $\kappa$ and every set $\mathscr{B}$ of $\mathcal{L}$-structures, if $|\mathscr{B}| = \kappa$ and every structure in $\mathscr{B}$ has cardinality less than $\kappa$, then an elementary embedding between two structures in $\mathscr{B}$ exists in some generic extension of $V$. We investigate connections between generic Vop\v{e}nka cardinals in models of ZFC and the number and complexity of $\aleph_1$-Suslin sets of reals in models of ZF. In particular, we show that ZFC + (there is a generic Vop\v{e}nka cardinal) is equiconsistent with ZF + $(2^{\aleph_1} \not\leq |S_{\aleph_1}|)$ where $S_{\aleph_1}$ is the pointclass of all $\aleph_1$-Suslin sets of reals, and also with ZF + $(S_{\aleph_1} = {\bfSigma}^1_2)$ + $(\Theta = \aleph_2)$ where $\Theta$ is the least ordinal that is not a surjective image of the~reals.
\end{abstract}

\maketitle

\section{Introduction}

A set of reals, meaning a subset of the Baire space $\omega^\omega$, is called ${\bfSigma}^1_1$ if it is the projection of a closed subset of $\omega^\omega \times \omega^\omega$. There are two natural ways to generalize this definition. First, we may begin with a closed subset of a higher-dimensional product $\omega^\omega \times \cdots \times \omega^\omega$ and alternate between projection and complementation to obtain the projective pointclasses ${\bfSigma}^1_1$, ${\bfPi}^1_1$, ${\bfSigma}^1_2$, ${\bfPi}^1_2$, and so on. Alternatively, we may project closed subsets of $\omega^\omega \times \kappa^\omega$ for $\kappa = \aleph_1$, $\aleph_2$, $\aleph_3$, and so on. The projection of a closed subset of $\omega^\omega \times \kappa^\omega$ is called \emph{$\kappa$-Suslin} and the pointclass of all $\kappa$-Suslin sets of reals is denoted by $S_\kappa$.  A set of reals that is $\kappa$-Suslin for some well-ordered cardinal $\kappa$ is called \emph{Suslin}. The relationship between the projective and Suslin hierarchies is already interesting at the level of ${\bfSigma}^1_2$ sets and $\aleph_1$-Suslin sets.

A couple of basic facts provable in ZF about the pointclass $S_{\aleph_1}$ of $\aleph_1$-Suslin sets are (1) $|S_{\aleph_1}| \le^* 2^{\aleph_1}$, meaning that there is a surjection from $\powerset(\aleph_1)$ to $S_{\aleph_1}$, and (2) ${\bfSigma}^1_2 \subset S_{\aleph_1}$. Fact (1) follows from the fact that the topological space $\omega^\omega \times \aleph_1^\omega$ has a basis of cardinality $\aleph_1$, so its closed subsets can be coded by subsets of $\aleph_1$. Fact (2) is due to Shoenfield (see Kanamori \cite[Theorem 13.14]{KanHigherInfinite}). Two natural questions to consider next are whether $2^{\aleph_1} \le |S_{\aleph_1}|$, in other words whether there is an injection from $\powerset(\aleph_1)$ to $S_{\aleph_1}$, and whether $S_{\aleph_1} = {\bfSigma}^1_2$. We will briefly discuss the answers to these questions in ZFC and then move to the ZF context where the main results of this article will be obtained.

Assuming ZFC we have $2^{\aleph_1} \le |S_{\aleph_1}|$ and therefore $|S_{\aleph_1}| = 2^{\aleph_1}$. This follows from the fact that every well-ordered set of reals $A$ is $|A|$-Suslin, so fixing a set of reals of cardinality $\aleph_1$, its subsets form a family of $2^{\aleph_1}$ distinct $\aleph_1$-Suslin sets. Assuming ZFC + CH we have $S_{\aleph_1} \ne {\bfSigma}^1_2$ because the aforementioned fact now shows that every set of reals is $\aleph_1$-Suslin, whereas not every set of reals can be ${\bfSigma}^1_2$ because $|{\bfSigma}^1_2| \le^* 2^{\aleph_0}$ by definition. On the other hand, assuming ZFC + $\neg$CH + MA + $(\aleph_1^L = \aleph_1)$ we have $S_{\aleph_1} = {\bfSigma}^1_2$ by Martin and Solovay \cite[Section 3]{MarSolInternalCohenExtensions}.

Note that the above proof of $2^{\aleph_1} \le |S_{\aleph_1}|$ in ZFC only requires ZF + $(\aleph_1 \le 2^{\aleph_0})$. Therefore if we assume ZF + $(2^{\aleph_1} \not\le |S_{\aleph_1}|)$ then we have $\aleph_1 \not\le 2^{\aleph_0}$, which implies that $\aleph_1^V$ is a strong limit cardinal in every inner model of ZFC. If we also assume the axiom of countable choice then $\aleph_1^V$ is regular, so it is an inaccessible cardinal in every inner model of ZFC. We will show that in fact the consistency strength of the theory ZF + $(2^{\aleph_1} \not\leq |S_{\aleph_1}|)$  is strictly higher than the existence of an inaccessible cardinal, even without assuming countable choice.

The consistency strength of this theory can be described in terms of a variant of the generic Vop\v{e}nka principle.
First defined by Bagaria, Gitman, and Schindler \cite{BagGitSchGenericVopenka}, the generic Vop\v{e}nka principle says that for every first-order language $\mathcal{L}$ and every proper class $\mathscr{B}$ of $\mathcal{L}$-structures there are distinct structures $\mathcal{M}, \mathcal{M}' \in \mathscr{B}$ such that an elementary embedding of $\mathcal{M}$ into $\mathcal{M}'$ exists in some generic extension of $V$. It can be formalized in ZFC for proper classes definable from parameters, in which case it may be called the generic Vop\v{e}nka scheme, or in GBC (G\"{o}del--Bernays set theory with the axiom of global choice) for arbitrary proper classes as in Gitman and Hamkins \cite{GitHamOrdNotDelta2Mahlo}.

The generic Vop\v{e}nka principle and the variant of it that we will define below are instances of ``virtual'' large cardinal principles.  For more information on virtual large cardinals, see Gitman and Schindler \cite{GitSchVirtualLargeCardinals}.  For an application of virtual large cardinals to descriptive set theory involving universally Baire sets instead of $\aleph_1$-Suslin sets, see Schindler and Wilson~\cite{SchWilConPSPuB}.

Given two structures $\mathcal{M}$ and $\mathcal{M}'$ we will abbreviate the statement ``an elementary embedding of $\mathcal{M}$ into $\mathcal{M}'$ exists in some generic extension of $V$'' by the phrase ``there is a generic elementary embedding of $\mathcal{M}$ into $\mathcal{M}'$,'' which is a harmless abuse of notation. Note that by the absoluteness of elementary embeddability of countable structures (see Bagaria, Gitman, and Schindler \cite[Lemma 2.6]{BagGitSchGenericVopenka}) this statement is equivalent to the statement ``every generic extension of $V$ by $\Col(\omega,\mathcal{M})$ contains an elementary embedding of $\mathcal{M}$ into $\mathcal{M}'$.''

For a cardinal $\kappa$, we may define a local version of the generic Vop\v{e}nka principle by replacing sets and proper classes with sets of cardinality less than $\kappa$ and sets of cardinality $\kappa$ respectively:

\begin{definition}[ZFC]
 A \emph{generic Vop\v{e}nka-like cardinal} is an uncountable cardinal $\kappa$ such that for every first-order language $\mathcal{L}$ of cardinality less than $\kappa$ and every set $\mathscr{B}$ of $\mathcal{L}$-structures, if $\mathscr{B}$ has cardinality $\kappa$ and every structure in $\mathscr{B}$ has cardinality less than $\kappa$, then there are distinct structures $\mathcal{M}, \mathcal{M}' \in \mathscr{B}$ and a generic elementary embedding of $\mathcal{M}$ into $\mathcal{M}'$.
\end{definition}

In the context of this definition every $\mathcal{L}$-structure of cardinality less than $\kappa$ is isomorphic to a structure in $H_\kappa$, so we may as well assume $\mathscr{B} \subset H_\kappa$. Although the definition does not require inaccessibility of $\kappa$, it is easily seen to imply the strong limit property:

\begin{lemma}[ZFC] \label{lem:strong-limit}
 Every generic Vop\v{e}nka-like cardinal is a strong limit cardinal.
\end{lemma}
\begin{proof}
 Let $\kappa$ be an uncountable cardinal that is not strong limit.  Then $\kappa \le 2^\alpha$ for some cardinal $\alpha < \kappa$, and letting $\mathcal{L}$ be a first-order language with $\alpha$ unary predicate symbols we may obtain $\kappa$ distinct $\mathcal{L}$-structures with universe $\{0\}$. There cannot be a generic elementary embedding between two such structures, so $\kappa$ is not a generic Vop\v{e}nka-like cardinal.
\end{proof}

Singular generic Vop\v{e}nka-like cardinals can exist, and we will show that they can be completely characterized in terms of $\omega$-Erd\H{o}s cardinals as defined by Baumgartner \cite{BauIneffabilityII}.  (The role of $\omega$-Erd\H{o}s cardinals here is similar to their role in Wilson \cite{WilWeaklyRemarkable}.) The following characterization of singular generic Vop\v{e}nka-like cardinals will be proved in Section \ref{section:generic-Vopenka-like-and-Erdos}.

\begin{proposition}[ZFC] \label{prop:singular-equivalence}
 For every singular cardinal $\kappa$, the following statements are equivalent:
 \begin{enumerate}
  \item $\kappa$ is a generic Vop\v{e}nka-like cardinal.
  \item $\kappa$ is a limit of $\omega$-Erd\H{o}s cardinals.
 \end{enumerate}
\end{proposition}

Because $\omega$-Erd\H{o}s cardinals are relatively well-understood, we will focus instead on generic Vop\v{e}nka-like cardinals that are regular and therefore inaccessible:

\begin{definition}[ZFC]
 A \emph{generic Vop\v{e}nka cardinal} is a regular generic Vop\v{e}nka-like \mbox{cardinal}.
\end{definition}

Recall that a cardinal $\kappa$ is inaccessible if and only if the two-sorted structure $(V_\kappa , V_{\kappa+1};\mathord{\in})$ satisfies GBC.  Moreover, note that a cardinal $\kappa$ is a generic Vop\v{e}nka cardinal if and only if $(V_\kappa , V_{\kappa+1};\mathord{\in})$ satisfies GBC + the generic Vop\v{e}nka principle. Therefore the existence of a generic Vop\v{e}nka cardinal has higher consistency strength than the generic Vop\v{e}nka principle.

The following result gives an upper bound for the consistency strength of generic Vop\v{e}nka cardinals. It will also be proved in Section \ref{section:generic-Vopenka-like-and-Erdos}, using an argument of Gitman and Schindler \cite{GitSchVirtualLargeCardinals} to obtain generic elementary embeddings from the $\omega$-Erd\H{o}s property.

\begin{proposition}[ZFC] \label{prop:omega-erdos-is-gen-vop-lim-of-gen-vop}
 Every $\omega$-Erd\H{o}s cardinal is a generic Vop\v{e}nka cardinal and a limit of generic Vop\v{e}nka cardinals.
\end{proposition}

Propositions \ref{prop:singular-equivalence} and \ref{prop:omega-erdos-is-gen-vop-lim-of-gen-vop} together imply that every singular generic Vop\v{e}nka-like cardinal is a limit of generic Vop\v{e}nka cardinals.  In particular the least generic Vop\v{e}nka-like cardinal is regular and is therefore a generic Vop\v{e}nka cardinal.

Before proving our main theorem relating generic Vop\v{e}nka cardinals to $\aleph_1$-Suslin sets, we will prove a related but simpler result characterizing generic Vop\v{e}nka cardinals in terms of generically hereditary sets of structures, defined as follows.

\begin{definition}[ZFC]
 Let $\mathscr{A} \subset H_\kappa$ where $\kappa$ is an uncountable cardinal.  We say that $\mathscr{A}$ is \emph{generically hereditary} if there is a first-order language $\mathcal{L}$ of cardinality less than $\kappa$ such that
 \begin{enumerate}
  \item Every element of $\mathscr{A}$ is an $\mathcal{L}$-structure, and
  \item For all $\mathcal{L}$-structures $\mathcal{M},\mathcal{M}' \in H_\kappa$, if $\mathcal{M}' \in \mathscr{A}$ and there is a generic elementary embedding of $\mathcal{M}$ into $\mathcal{M}'$, then $\mathcal{M} \in \mathscr{A}$.
 \end{enumerate}
 (Although the definition depends on $\kappa$, it will always be clear from context.)
\end{definition}

The following result, which will be proved in section \ref{section:hereditary}, gives some equivalent conditions for the generic Vop\v{e}nka property of a cardinal in terms of the complexity and number of generically hereditary sets of structures. It is possible to prove more equivalences along these lines, but we will limit ourselves here to statements that have counterparts in our main equiconsistency theorem about $\aleph_1$-Suslin sets of reals.

\begin{proposition}[ZFC] \label{prop:hereditary}
 For every uncountable cardinal $\kappa$, the following statements are equivalent.
 \begin{enumerate}
  \item\label{item:hereditary-gen-vop} $\kappa$ is a generic Vop\v{e}nka cardinal.
  \item\label{item:hereditary-complexity} Every generically hereditary subset of $H_\kappa$ is ${\bfSigma}_1^{H_\kappa}$.
  \item\label{item:hereditary-fewer-than-2-H-kappa} There are fewer than $2^\kappa$ generically hereditary subsets of $H_\kappa$.
 \end{enumerate}
\end{proposition}

Our main theorem relating generic Vop\v{e}nka cardinals to the complexity and number of $\aleph_1$-Suslin sets is stated below. It will be proved in section \ref{section:proof-of-thm}. In the theorem statement, $\Theta$ denotes the least ordinal that is not a surjective image of the reals, sometimes called the \emph{Lindenbaum number} of the reals. Note that $\Theta$ is a cardinal greater than or equal to $\aleph_2$. The axiom of choice implies $\Theta = (2^{\aleph_0})^+$, so the statement $\Theta = \aleph_2$ is equivalent to CH in ZFC.

\begin{theorem}\label{thm:main-equicon}
 The following theories are equiconsistent.
 \begin{enumerate}
  \item\label{item:gen-vopenka} ZFC + there is a generic Vop\v{e}nka cardinal.
  \item\label{item:aleph-1-suslin-is-sigma-1-2} ZF + $(S_{\aleph_1} = {\bfSigma}^1_2)$ + $(\Theta = \aleph_2)$.
  \item\label{item:no-injection-into-S-aleph-1} ZF + $(2^{\aleph_1} \not\leq |S_{\aleph_1}|)$.
 \end{enumerate}
\end{theorem}

The proof of Theorem \ref{thm:main-equicon} is outlined as follows. First, we will show that if $\kappa$ is a generic Vop\v{e}nka cardinal in $V$ then it is a generic Vop\v{e}nka cardinal in $L$ and theory \ref{item:aleph-1-suslin-is-sigma-1-2} holds in $L(\mathbb{R})^{L[G]}$ where $G \subset \Col(\omega,\mathord{<}\kappa)$ is an $L$-generic filter. Second, we will show that theory \ref{item:aleph-1-suslin-is-sigma-1-2} implies theory \ref{item:no-injection-into-S-aleph-1}. Finally, we will show that if theory \ref{item:no-injection-into-S-aleph-1} holds then $\aleph_1^V$ is a generic Vop\v{e}nka-like cardinal in $L$ (and in every inner model of ZFC) and so by Propositions \ref{prop:singular-equivalence} and \ref{prop:omega-erdos-is-gen-vop-lim-of-gen-vop} the least generic Vop\v{e}nka-like cardinal in $L$ is a generic Vop\v{e}nka cardinal in $L$.

\begin{remark}
 We can obtain some more equiconsistencies with the theory
 ZFC + (there is a generic Vop\v{e}nka cardinal) along the lines of Theorem \ref{thm:main-equicon} without much additional effort.
 
 First, adding DC to theories \ref{item:aleph-1-suslin-is-sigma-1-2} and \ref{item:no-injection-into-S-aleph-1} yields equiconsistent theories because DC holds in the model $L(\mathbb{R})^{L[G]}$ used in the proof that $\Con(\ref{item:gen-vopenka})$ implies $\Con(\ref{item:aleph-1-suslin-is-sigma-1-2})$.

 Second, replacing $S_{\aleph_1}$ with the pointclass of all Suslin sets in theories \ref{item:aleph-1-suslin-is-sigma-1-2} and \ref{item:no-injection-into-S-aleph-1} yields equiconsistent theories.  This is because $\Theta = \aleph_2$ implies that every Suslin set of reals is $\aleph_1$-Suslin.  (We can represent a closed subset of $\omega^\omega \times \kappa^\omega$, where $\kappa$ is a well-ordered cardinal, as the set $[T]$ of all infinite branches through a tree $T$ on $\omega \times \kappa$, and observe that $T$ has a subtree $T_0$ such that $[T]$ and $[T_0]$ have the same projection and $|T_0| \le^* 2^{\aleph_0}$.)  Alternatively, we may observe that the proof that $\Con(\ref{item:gen-vopenka})$ implies $\Con(\ref{item:aleph-1-suslin-is-sigma-1-2})$ applies equally well to all Suslin~sets.
\end{remark}

We end this section with a couple of remaining questions.

First, note that the analogy between Proposition \ref{prop:hereditary} and Theorem \ref{thm:main-equicon} is marred by the inclusion of the hypothesis $\Theta = \aleph_2$ in theory \ref{item:aleph-1-suslin-is-sigma-1-2} of Theorem \ref{thm:main-equicon}. This hypothesis is needed because the consistency strength of ZFC + $(S_{\aleph_1} = {\bfSigma}^1_2)$ alone is trivial by the previously mentioned result of Martin and Solovay. However, if we replace $S_{\aleph_1}$ by the pointclass of all Suslin sets in theory \ref{item:aleph-1-suslin-is-sigma-1-2} then the necessity of the hypothesis $\Theta = \aleph_2$ is not clear:

\begin{question}\label{question:no-CH}
 What is the consistency strength of the theory ZF + every Suslin set of reals is ${\bfSigma}^1_2$?
\end{question}

One could also ask this question with DC added to the theory.

Second, note that our proof that $\Con(\ref{item:no-injection-into-S-aleph-1})$ implies $\Con(\ref{item:gen-vopenka})$ in Theorem \ref{thm:main-equicon} only uses our results involving $\omega$-Erd\H{o}s cardinals, namely Propositions \ref{prop:singular-equivalence} and \ref{prop:omega-erdos-is-gen-vop-lim-of-gen-vop}, in the case that theory \ref{item:no-injection-into-S-aleph-1} holds and $\aleph_1$ is singular. I do not know if this case can actually occur:

\begin{question}
 Is the theory ZF + $(2^{\aleph_1} \not\leq |S_{\aleph_1}|)$ + $(\aleph_1 \text{ is singular})$ consistent?
\end{question}

One could also ask this question with the pointclass $S_{\aleph_1}$ replaced by the pointclass of all Suslin sets.

\section{Generic Vop\v{e}nka-like cardinals and Erd\H{o}s cardinals}\label{section:generic-Vopenka-like-and-Erdos}

We work in ZFC for the duration of this section. Propositions \ref{prop:singular-equivalence} and \ref{prop:omega-erdos-is-gen-vop-lim-of-gen-vop} will be obtained as consequences of the following collection of lemmas.

\begin{lemma}\label{lem:sing-gen-Vop-like-is-lim-of-omega-erdos}
 Every singular generic Vop\v{e}nka-like cardinal is a limit of $\omega$-Erd\H{o}s cardinals.
\end{lemma}
\begin{proof}
 Let $\lambda$ be a singular generic Vop\v{e}nka-like cardinal and suppose toward a contradiction that for some cardinal $\alpha < \lambda$ there is no $\omega$-Erd\H{o}s cardinal between $\alpha$ and $\lambda$.  Then for every cardinal $\eta < \lambda$ we have $\eta \not\to (\omega)^{\mathord{<}\omega}_\alpha$ because otherwise the least cardinal $\eta$ such that $\eta \to (\omega)^{\mathord{<}\omega}_\alpha$ would be an $\omega$-Erd\H{o}s cardinal between $\alpha$ and $\lambda$.
 
 Let $(\eta_i : i < \cf(\lambda))$ be an increasing sequence of infinite cardinals above $\alpha$ and cofinal in $\lambda$, and for every ordinal $i < \cf(\lambda)$ choose a function $f_i: [\eta_i]^{\mathord{<}\omega} \to \alpha$ witnessing $\eta_i \not\to (\omega)^{\mathord{<}\omega}_\alpha$, meaning that $f_i$ has no homogeneous set of order type $\omega$. For every ordinal $i < \cf(\lambda)$ and every ordinal $\beta < \eta_i$, we define the structure
 \[\mathcal{M}_{i, \beta} = ( L_{\eta_i}; \mathord{\in}, f_i, R_{i'}^{\mathcal{M}_{i, \beta}}, \alpha', \beta)_{\alpha' \le \alpha, i' < \cf(\lambda)}\]
 in the language with a binary predicate symbol for set membership, a unary function symbol for $f_i$, a nullary predicate symbol $R_{i'}$ for each ordinal $i' < \cf(\lambda)$, a constant symbol for each ordinal $\alpha' \le \alpha$, and a constant symbol for $\beta$, where $R_{i}^{\mathcal{M}_{i, \beta}} = \top$ and $R_{i'}^{\mathcal{M}_{i, \beta}} = \bot$ for all $i' \ne i$.
 
 This language has cardinality $\max(\alpha, \cf(\lambda))< \lambda$, each structure $\mathcal{M}_{i, \beta}$ has cardinality $\eta_i < \lambda$, and the number of such structures is $\sum_{i < \cf(\lambda)} \eta_i = \lambda$, so because $\lambda$ is a generic Vop\v{e}nka-like cardinal there are distinct pairs of ordinals $(i_0,\beta_0)$ and $(i_1, \beta_1)$ and a generic elementary embedding 
 \[ j : \mathcal{M}_{i_0, \beta_0} \to \mathcal{M}_{i_1, \beta_1}.\]
 We must have $i_0 = i_1$ because $j$ preserves the interpretation of the nullary predicate symbol $R_{i'}$ for every $i' < \cf(\lambda)$.  Therefore $\beta_0 \ne \beta_1$, and because $j(\beta_0) = \beta_1$ it follows that $j$ has a critical point. Because $j(\alpha') = \alpha'$ for all $\alpha' \le \alpha$, we have $\crit(j) > \alpha$. Defining $i = i_0 = i_1$, we may consider $j$ as a generic elementary embedding
 \[ j: ( L_{\eta_i}; \mathord{\in}, f_i) \to ( L_{\eta_i}; \mathord{\in}, f_i). \]
 
 Let $(\kappa_n : n< \omega)$ be the critical sequence of $j$, defined by $\kappa_n = j^n(\crit(j))$. By the elementarity of $j$ and the fact that $\ran(f_i) \subset \alpha < \crit(j)$, for all $n<\omega$ we have
 \[f_i(\kappa_0, \ldots,\kappa_{n-1}) = f_i(\kappa_1,\ldots,\kappa_n).\]
 Arguing further along these lines as in Silver \cite[Section 2]{SilLargeCardinalInL} shows that the set $\{\kappa_n : n< \omega\}$ is homogeneous for $f_i$.  This particular homogeneous set might not exist in $V$ because $j$ might not exist in $V$, but the absoluteness argument of Silver \cite[Section 1]{SilLargeCardinalInL} shows that some homogeneous set for $f_i$ of order type $\omega$ exists in $V$, contradicting our choice of $f_i$.
\end{proof}

A cardinal $\kappa$ is called a \emph{virtual rank-into-rank cardinal} if there is an ordinal $\lambda > \kappa$ and a generic elementary embedding $j : (V_\lambda;\mathord{\in}) \to (V_\lambda; \mathord{\in})$ with $\crit(j) = \kappa$. The following lemma is essentially due to Gitman and Schindler \cite{GitSchVirtualLargeCardinals}. Although they stated it only for the least $\omega$-Erd\"{o}s cardinal, their argument applies to any $\omega$-Erd\"{o}s cardinal as defined by Baumgartner~\cite{BauIneffabilityII}.

\begin{lemma}[{Gitman and Schindler \cite[Theorem 4.17]{GitSchVirtualLargeCardinals}}] \label{lem:omega-Erdos-is-lim-of-v-rank-into-rank}
 Every $\omega$-Erd\H{o}s cardinal is a limit of virtual rank-into-rank cardinals.
\end{lemma}

We will need the following generalization, which can be proved by a similar argument. The proof of this generalization is implicit in Wilson \cite[Lemma 2.5]{WilWeaklyRemarkable}, so we will not repeat it here.
% [TO DO: it should be made explicit there.]

\begin{lemma}\label{lem:gen-emb-from-omega-erdos}
 Assume that $\eta$ is an $\omega$-Erd\H{o}s cardinal. Then for every cardinal $\alpha < \eta$, every ordinal $\lambda \ge \eta$, and every set $A \subset V_\lambda$, there is a generic elementary embedding
 \[j : (V_\lambda;\mathord{\in}, A) \to (V_\lambda; \mathord{\in}, A)\]
 such that $\alpha < \crit(j) < \eta$.
\end{lemma}

\begin{lemma}\label{lem:omega-Erdos-or-lim-of-omega-Erdos-is-gen-Vop-like}
 Assume that $\lambda$ is an $\omega$-Erd\H{o}s cardinal or a limit of $\omega$-Erd\H{o}s cardinals. Then $\lambda$ is a generic Vop\v{e}nka-like cardinal.
\end{lemma}
\begin{proof}
 Let $\vec{\mathcal{M}}$ be a $\lambda$-sequence of structures for the same first-order language $\mathcal{L}$ such that $|\mathcal{L}| < \lambda$ and every structure on the sequence is in $H_\lambda$. We will show that there is a generic elementary embedding between two structures on the sequence. Let $\eta$ be the least $\omega$-Erd\H{o}s cardinal greater than $|\mathcal{L}|$. Our hypothesis on $\lambda$ implies that $\eta \le \lambda$ and also that $H_\lambda = V_\lambda$. By Lemma \ref{lem:gen-emb-from-omega-erdos} there is a generic elementary embedding 
 \[j : (V_\lambda;\mathord{\in}, \vec{\mathcal{M}}) \to (V_\lambda; \mathord{\in}, \vec{\mathcal{M}})\]
 such that $|\mathcal{L}| < \crit(j) < \lambda$.  Let $\kappa = \crit(j)$. Because $j$ preserves $\vec{\mathcal{M}}$ we have $j(\vec{\mathcal{M}}(\kappa)) = \vec{\mathcal{M}}(j(\kappa))$, and because $\crit(j) > |\mathcal{L}|$ it follows that the restriction $j \restriction \vec{\mathcal{M}}(\kappa)$ is a generic elementary embedding of the structure $\vec{\mathcal{M}}(\kappa)$ into $\vec{\mathcal{M}}(j(\kappa))$.
\end{proof}

A weaker version of the following lemma with ``$\kappa$ is a generic Vop\v{e}nka cardinal'' replaced by ``$V_\kappa$ satisfies the generic Vop\v{e}nka principle for classes definable from parameters'' follows from Bagaria, Gitman, and Schindler \cite[Proposition 3.10 and Theorem 5.2]{BagGitSchGenericVopenka}. A non-virtual version of the following lemma with ``generic Vop\v{e}nka cardinal'' replaced by ``Vop\v{e}nka cardinal'' and ``virtual rank-into-rank cardinal'' replaced by ``$\omega$-huge cardinal'' follows from results of Solovay, Reinhardt, and Kanamori \cite[Section 8]{SolReiKanStrongAxioms}.

\begin{lemma}\label{lem:v-rank-into-rank-implies-gen-vop}
 Every virtual rank-into-rank cardinal is a generic Vop\v{e}nka cardinal.
\end{lemma}
\begin{proof}
 Let $\kappa$ be a virtual rank-into-rank cardinal. Then there is an ordinal $\lambda > \kappa$ and a generic elementary embedding $j : (V_\lambda;\mathord{\in}) \to (V_\lambda; \mathord{\in})$ such that $\crit(j) = \kappa$. Clearly $\kappa$ is an inaccessible cardinal.  Let $\vec{\mathcal{M}}$ be a $\kappa$-sequence of structures for the same first-order language $\mathcal{L}$ such that $|\mathcal{L}| < \kappa$ and every structure on the sequence is in $H_\kappa$.
 
 Because $\crit(j) > |\mathcal{L}|$ it follows that $j(\vec{\mathcal{M}})$ is a $j(\kappa)$-sequence of $\mathcal{L}$-structures, $j(j(\vec{\mathcal{M}}))$ is a $j(j(\kappa))$-sequence of $\mathcal{L}$-structures, and we have a generic elementary embedding
 \[ j \restriction j(\vec{\mathcal{M}})(\kappa) : j(\vec{\mathcal{M}})(\kappa) \to j(j(\vec{\mathcal{M}})(\kappa)) = j(j(\vec{\mathcal{M}}))(j(\kappa)).\]
 We have $j(\vec{\mathcal{M}}) \restriction \kappa = \vec{\mathcal{M}}$
 because $\crit(j) = \kappa$ and each structure on $\vec{\mathcal{M}}$ is in $H_\kappa$, so by the elementarity of $j$ it follows that 
 \[j(j(\vec{\mathcal{M}})) \restriction j(\kappa) = j(\vec{\mathcal{M}})\]
 and in particular we have
 \[j(j(\vec{\mathcal{M}}))(\kappa) = j(\vec{\mathcal{M}})(\kappa).\]
 Therefore $j \restriction j(\vec{\mathcal{M}})(\kappa)$ is a generic elementary embedding between two structures on the sequence $j(j(\vec{\mathcal{M}}))$, namely $j(j(\vec{\mathcal{M}}))(\kappa)$ and $j(j(\vec{\mathcal{M}}))(j(\kappa))$. Applying the elementarity of $j$ twice, it follows that there is a generic elementary embedding between two structures on the sequence $\vec{\mathcal{M}}$.
\end{proof}

Because the statement ``$\kappa$ is a generic Vop\v{e}nka cardinal'' is $\Pi^1_1$ over $(V_\kappa, V_{\kappa+1}; \mathord{\in})$ and virtual rank-into-rank cardinals are $\Pi^1_1$-indescribable, Lemma \ref{lem:v-rank-into-rank-implies-gen-vop} furthermore implies that every virtual rank-into-rank cardinal is a limit of generic Vop\v{e}nka cardinals.

Proposition \ref{prop:singular-equivalence} follows from Lemmas \ref{lem:sing-gen-Vop-like-is-lim-of-omega-erdos} and \ref{lem:omega-Erdos-or-lim-of-omega-Erdos-is-gen-Vop-like}. Proposition \ref{prop:omega-erdos-is-gen-vop-lim-of-gen-vop} follows from Lemmas \ref{lem:omega-Erdos-is-lim-of-v-rank-into-rank}, \ref{lem:omega-Erdos-or-lim-of-omega-Erdos-is-gen-Vop-like}, \ref{lem:v-rank-into-rank-implies-gen-vop}, and the fact that every $\omega$-Erd\H{o}s cardinal is regular.

\section{Generic Vop\v{e}nka cardinals and hereditary sets of structures} \label{section:hereditary}
 
In this section we will prove Proposition \ref{prop:hereditary}, which characterizes the generic Vop\v{e}nka cardinals in terms of the definability and number of generically hereditary sets of structures.  We work in ZFC for the duration of the proof.
 
 \begin{proof}[Proof of Proposition \ref{prop:hereditary}]\ \\

 \eqref{item:hereditary-gen-vop} implies \eqref{item:hereditary-complexity}: Let $\kappa$ be a generic Vop\v{e}nka cardinal, let $\mathcal{L}$ be a first-order language of cardinality less than $\kappa$, and let $\mathscr{A} \subset H_\kappa$ be a generically hereditary set of $\mathcal{L}$-structures. We will show that $\mathscr{A}$ is ${\bfSigma}_1$-definable over $H_\kappa$.
 
 Let $\mathscr{B}$ be the set of all $\mathcal{L}$-structures in $H_\kappa \setminus \mathscr{A}$. By the definition of ``generically hereditary'' the set $\mathscr{A}$ is downward closed in $H_\kappa$ with respect to generic elementary embeddability, so the set $\mathscr{B}$ is upward closed in $H_\kappa$ with respect to generic elementary embeddability. We claim that $\mathscr{B}$ is generated as the upward closure of some subset $\mathscr{B}_0 \subset \mathscr{B}$ such that $|\mathscr{B}_0| < \kappa$.
 
 Assume toward a contradiction that $\mathscr{B}$ is not generated in this way by any such small subset $\mathscr{B}_0$.  Then by transfinite recursion using the axiom of choice we may obtain a $\kappa$-sequence $\vec{\mathcal{M}}$ of distinct $\mathcal{L}$-structures in $\mathscr{B}$ such that whenever $\alpha < \alpha' < \kappa$ there is no generic elementary embedding of $\vec{\mathcal{M}}(\alpha)$ into $\vec{\mathcal{M}}(\alpha')$. By replacing each structure $\vec{\mathcal{M}}(\alpha)$  with a structure coding the pair of structures $(\vec{\mathcal{M}}(\alpha),(\alpha ; \mathord{\in}))$ if necessary, we may assume there is also no generic elementary embedding in the reverse direction. The existence of such a sequence of structures contradicts the generic Vop\v{e}nka property of $\kappa$, as desired.
 
 Now fix a small subset $\mathscr{B}_0$ generating $\mathscr{B}$ as above. Because $\mathscr{B}_0 \subset H_\kappa$, $|\mathscr{B}_0| < \kappa$, and $\kappa$ is regular we have $\mathscr{B}_0 \in H_\kappa$, so we can use it as a parameter in a definition over $H_\kappa$. For every $\mathcal{M} \in H_\kappa$ we have $\mathcal{M} \in \mathscr{A}$ if and only if $\mathcal{M}$ is an $\mathcal{L}$-structure and for every $\mathcal{M}_0 \in \mathscr{B}_0$ there is no generic elementary embedding of $\mathcal{M}_0$ into $\mathcal{M}$. To show that $\mathscr{A}$ is ${\bfSigma}_1^{H_\kappa}$ it will therefore suffice to express the nonexistence of a generic elementary embedding of $\mathcal{M}_0$ into $\mathcal{M}$ by a $\Sigma_1$ formula over $H_\kappa$, or equivalently over $V$ because $\kappa$ is inaccessible.\footnote{We can also express the nonexistence of a generic elementary embedding between two structures by a $\Pi_1$ formula, showing that $\mathscr{A}$ is ${\bfPi}_1^{H_\kappa}$, but this is not relevant to our main result on $\aleph_1$-Suslin sets.}
 
 By Bagaria, Gitman, and Schindler \cite[Proposition 4.1]{BagGitSchGenericVopenka}, the existence of a generic elementary embedding of $\mathcal{M}_0$ into $\mathcal{M}$ is equivalent to the existence of a winning strategy for player II in the game $G(\mathcal{M}_0,\mathcal{M})$ defined as follows. In round $n$ of the game, player I chooses an element $x_n \in \mathcal{M}_0$ and then player II chooses an element $y_n \in \mathcal{M}$. Player II survives round $n$ if and only if the type of $(x_0,\ldots,x_n)$ in $\mathcal{M}_0$ equals the type of $(y_0,\ldots,y_n)$ in $\mathcal{M}$; otherwise the position $(x_0,y_0,\ldots,x_n,y_n)$ is considered to be an immediate loss for player II. To win, player II must survive all $\omega$ rounds.
 
 This game is determined by the Gale--Stewart theorem, so the nonexistence of a generic elementary embedding of $\mathcal{M}_0$ into $\mathcal{M}$ is equivalent to the existence of a winning strategy for player I. Because player I's payoff set is open, this condition is furthermore equivalent to the existence of a nonempty subtree $S$ of the game tree with all of the following properties:
 
 \begin{itemize}
  \item For every position $\sigma \in S$ of even length there is some $x\in \mathcal{M}_0$ such that $\sigma^\frown x \in S$.
  \item For every position $\sigma \in S$ of odd length and every $y\in \mathcal{M}$, either $\sigma^\frown y$ is an immediate loss for player II or $\sigma^\frown y \in S$.
  \item $S$ is wellfounded.
 \end{itemize}
 Because wellfoundedness is witnessed by rank functions, this condition for the nonexistence of a generic elementary embedding can be expressed by a $\Sigma_1$ formula.
 
 \eqref{item:hereditary-complexity} implies \eqref{item:hereditary-fewer-than-2-H-kappa}: Let $\kappa$ be an uncountable cardinal and assume that every generically hereditary subset of $H_\kappa$ is ${\bfSigma}_1^{H_\kappa}$. Because the number of possible parameters for $\Sigma_1$ definitions over $H_\kappa$ is $|H_\kappa| = 2^{\mathord{<}\kappa}$, it follows that the number of generically hereditary subsets of $H_\kappa$ is at most $2^{\mathord{<}\kappa}$.  We will show that this weak upper bound implies the desired strict upper bound, namely that the number of generically hereditary subsets of $H_\kappa$ is less than $2^\kappa$. This implication is trivial if $2^{\mathord{<}\kappa} < 2^\kappa$, so we assume toward a contradiction that $2^{\mathord{<}\kappa} = 2^\kappa$.  
 
 We have $\operatorname{cf}(2^\kappa) > \kappa$ by K\"{o}nig's theorem, so $2^{\mathord{<}\kappa} = 2^\kappa$ implies $2^\alpha =  2^\kappa$ for some cardinal $\alpha < \kappa$. Our upper bound on the number of generically hereditary subsets of $H_\kappa$ therefore becomes $2^\alpha$. We will obtain a contradiction by showing that (in general) for every cardinal $\alpha < \kappa$ there are at least $2^{2^{\alpha}}$ generically hereditary subsets of $H_\kappa$.  Letting $\mathcal{L}$ be a first-order language with $\alpha$ nullary predicate symbols, the number of equivalence classes of $\mathcal{L}$-structures in $H_\kappa$ with respect to elementary equivalence is $2^\alpha$. For every set of such equivalence classes, its union is generically hereditary, so the number of generically hereditary sets of $\mathcal{L}$-structures is at least $2^{2^\alpha}$, giving the desired contradiction.
 
 \eqref{item:hereditary-fewer-than-2-H-kappa} implies \eqref{item:hereditary-gen-vop}: Let $\kappa$ be an uncountable cardinal and assume that statement \ref{item:hereditary-gen-vop} fails, meaning that $\kappa$ is not a generic Vop\v{e}nka cardinal.  There are two possibilities: either $\kappa$ is not a generic Vop\v{e}nka-like cardinal or $\kappa$ is a singular generic Vop\v{e}nka-like cardinal, and in each case we will show that the number of generically hereditary subsets of $H_\kappa$ is at least $2^\kappa$.
 
 First we consider the case in which $\kappa$ is not a generic Vop\v{e}nka-like cardinal. Then there is a set of structures $\mathscr{B} \subset H_\kappa$ for the same first-order language such that $|\mathscr{B}| = \kappa$ and there is no generic elementary embedding between two distinct structures in $\mathscr{B}$. For every subset $\mathscr{A} \subset \mathscr{B}$ we let $\mathscr{A}^*$ be the downward closure of $\mathscr{A}$ in $H_\kappa$, meaning the set of all structures in $H_\kappa$ admitting a generic elementary embedding into some structure in $\mathscr{A}$.  Note that $\mathscr{A}^*$ is generically hereditary by definition.
 
 Because $|\powerset(\mathscr{B})| = 2^\kappa$ it will suffice to show that the function $\mathscr{A} \mapsto \mathscr{A}^*$ is injective. Let $\mathscr{A}_0$ and $\mathscr{A}_1$ be distinct subsets of $\mathscr{B}$. We may assume without loss of generality that there is a structure $\mathcal{M} \in \mathscr{A}_0 \setminus \mathscr{A}_1$.  Then clearly $\mathcal{M} \in \mathscr{A}_0^*$. On the other hand, $\mathcal{M}$ cannot be in $\mathscr{A}_1^*$: it is not in $\mathscr{A}_1$ and it cannot be generically elementarily embedded into any element of $\mathscr{A}_1$ because no element of $\mathscr{B}$ can be generically elementarily embedded into any other element of $\mathscr{B}$. Therefore we have $\mathscr{A}_0^* \ne \mathscr{A}_1^*$ as~desired.
 
 Next we consider the case in which $\kappa$ is a singular generic Vop\v{e}nka-like cardinal. Then $\kappa$ is a singular strong limit cardinal by Lemma \ref{lem:strong-limit}, so $2^\kappa = \kappa^{\cf(\kappa)}$ and it will suffice to show that there are at least $\kappa^{\cf(\kappa)}$ generically hereditary subsets of $H_\kappa$.
 
 Let $\mathcal{L}$ be the language with a binary predicate symbol $\in$ for set membership and a nullary predicate symbol $R_i$ for every ordinal $i < \cf(\kappa)$. For all ordinals $i < \cf(\kappa)$ and $\beta < \kappa$ we define the $\mathcal{L}$-structure
 \[\mathcal{M}_{i,\beta} = (\beta ; \mathord{\in}, R^{\mathcal{M}_{i,\beta}}_{i'})_{i' < \cf(\kappa)}\]
 where $R^{\mathcal{M}_{i,\beta}}_i = \top$ and $R^{\mathcal{M}_{i,\beta}}_{i'} = \bot$ for all $i' \ne i$. For every function $f \in \kappa^{\cf(\kappa)}$ we may define a corresponding set of $\mathcal{L}$-structures
 \[\mathscr{A}_f = \{\mathcal{M}_{i,f(i)} : i <\cf(\kappa)\}.\]
 Let $\mathscr{A}_f^*$ be the downward closure of $A_f$ in $H_\kappa$ with respect to generic elementary embeddability. Note that for all ordinals $i,i' < \cf(\kappa)$ and $\beta, \beta' < \kappa$, there is a generic elementary embedding from $\mathcal{M}_{i,\beta}$ into $\mathcal{M}_{i',\beta'}$ if and only if $i = i'$ and $\beta \le \beta'$. Using this fact it is easy to see that for any two distinct functions $f_0,f_1 \in \kappa^{\cf(\kappa)}$ we have $\mathscr{A}_{f_0}^* \ne \mathscr{A}_{f_1}^*$, so the number of generically hereditary subsets of $H_\kappa$ is at least $\kappa^{\cf(\kappa)}$ as desired. The proof of Proposition \ref{prop:hereditary} is complete.
\end{proof}

\section{Proof of Theorem \ref{thm:main-equicon}} \label{section:proof-of-thm}
 
In this section we will prove our main equiconsistency result relating generic Vop\v{e}nka cardinals in models of ZFC to the complexity and number of $\aleph_1$-Suslin sets in models of ZF. First we will prove two lemmas showing that the generic Vop\v{e}nka property is absolute to inner models and small forcing extensions respectively.

\begin{lemma}[ZFC] \label{lem:downward-absolute}
 Let $\kappa$ be a generic Vop\v{e}nka cardinal and let $W$ be an inner model of ZFC.  Then $\kappa$ is a generic Vop\v{e}nka cardinal in $W$.
\end{lemma}
\begin{proof}
 Clearly the inaccessibility of $\kappa$ is downward absolute to $W$. Now in $W$ let $\mathcal{L}$ be a first-order language of cardinality less than $\kappa$ and let $\mathscr{B}$ be a set of $\mathcal{L}$-structures such that $|\mathscr{B}| = \kappa$ and every structure in $\mathscr{B}$ has cardinality less than $\kappa$. Then these properties of $\mathscr{B}$ hold in $V$ as well, so by the generic Vop\v{e}nka property of $\kappa$ there are distinct structures $\mathcal{M}, \mathcal{M}' \in \mathscr{B}$ such that an elementary embedding of $\mathcal{M}$ into $\mathcal{M}'$ exists in some generic extension $V[G]$. By the absoluteness of elementary embeddability of countable structures, an elementary embedding of $\mathcal{M}$ into $\mathcal{M}'$ exists in $W[H]$ where $H \subset \Col(\omega,\mathcal{M})$ is a $V[G]$-generic filter, so the statement ``there is a generic elementary embedding of $\mathcal{M}$ into $\mathcal{M}'$'' holds in $W$.
\end{proof}

To show that the generic Vop\v{e}nka property is preserved by small forcing it is convenient to use a result of 
Gitman and Hamkins \cite[Theorem 7]{GitHamOrdNotDelta2Mahlo}, which says that the generic Vop\v{e}nka principle is equivalent in GBC to the statement ``for every class $A$ there is a proper class of weakly virtually $A$-extendible cardinals'' where a cardinal $\alpha$ is called \emph{weakly virtually $A$-extendible} if for every ordinal $\lambda > \alpha$ there is an ordinal $\theta$ and a generic elementary embedding 
\[j : (V_\lambda;\mathord{\in}, A \cap V_\lambda) \to (V_\theta;\mathord{\in}, A \cap V_\theta)\]
such that $\crit(j) = \alpha$.
 
\begin{lemma}[ZFC] \label{lem:small-forcing}
 Let $\kappa$ be a generic Vop\v{e}nka cardinal and let $V[G]$ be a generic extension of $V$ by a poset $\mathbb{P} \in V_\kappa$. Then $\kappa$ is a generic Vop\v{e}nka cardinal in $V[G]$.
\end{lemma}
\begin{proof}
 The structures $(V_{\kappa}, V_{\kappa+1}; \mathord{\in})$ and $(V_{\kappa}[G], V_{\kappa+1}[G]; \mathord{\in})$ satisfy GBC because $\kappa$ is inaccessible in $V$ and remains inaccessible in $V[G]$, so we may apply the aforementioned theorem of Gitman and Hamkins in these structures. Our hypothesis therefore implies that for every set $A \in V_{\kappa+1}$, $\kappa$ is a limit of cardinals that are weakly $A$-extendible in $(V_{\kappa}, V_{\kappa+1}; \mathord{\in})$, and our desired conclusion will follow if we can show that for every set $A \in V_{\kappa+1}[G]$, $\kappa$ is a limit of cardinals that are weakly $A$-extendible in $(V_{\kappa}[G], V_{\kappa+1}[G]; \mathord{\in})$.
 
 Let $A \in V_{\kappa+1}[G]$ and take a $\mathbb{P}$-name $\dot{A} \in V_{\kappa+1}$ such that $\dot{A}_G = A$. Let $\alpha$ be a cardinal such that $\rank(\mathbb{P}) < \alpha < \kappa$ and $\alpha$ is weakly virtually $\dot{A}$-extendible in $(V_{\kappa}, V_{\kappa+1}; \mathord{\in})$. It will suffice to show that $\alpha$ is weakly virtually $A$-extendible in $(V_{\kappa}[G], V_{\kappa+1}[G]; \mathord{\in})$.
 
 Let $\lambda$ be an ordinal such that $\alpha < \lambda < \kappa$. Increasing $\lambda$ if necessary, we may assume without loss of generality that it is a limit ordinal.  By our assumption on $\alpha$ there is an ordinal $\theta < \kappa$ and a generic elementary embedding 
 \[j : (V_\lambda;\mathord{\in}, \dot{A} \cap V_\lambda) \to (V_\theta;\mathord{\in}, \dot{A} \cap V_\theta)\]
 with $\crit(j) = \alpha$. Because $\rank(\mathbb{P}) < \alpha$ we can extend $j$ to an elementary embedding 
 \[\jhat : (V_\lambda[G];\mathord{\in}, A\cap V_\lambda[G]) \to (V_\theta[G];\mathord{\in}, A \cap V_\theta[G])\]
 with $\crit(\jhat) = \alpha$ by defining $\jhat(\tau_G) = j(\tau)_G$ for every $\mathbb{P}$-name $\tau \in V_\lambda$.  Then $\jhat$ witnesses the weak virtual $A$-extendibility of $\alpha$ in $(V_{\kappa}[G], V_{\kappa+1}[G]; \mathord{\in})$ with respect to $\lambda$.
\end{proof}

Next we will prove two lemmas in ZF relating Suslin sets in $V$ to the generic Vop\v{e}nka property of $\aleph_1^V$ in inner models of ZFC. A useful fact when dealing with Suslin sets is that for every ordinal $\lambda$, a subset of $\omega^\omega \times \lambda^\omega$ is closed if and only if it has the form $[T]$ for some tree $T$ on $\omega \times \lambda$ where $[T]$ is the set of all infinite branches of $T$.  Therefore a set of reals is $\lambda$-Suslin if and only if it has the form $\p[T]$ for some tree on $\omega \times \lambda$ where $\p$ denotes the first coordinate projection.

\begin{lemma}[ZF] \label{lem:suslin-is-sigma-1-2}
 Assume that $\aleph_1^V$ is a generic Vop\v{e}nka cardinal in $L[x]$ where $x$ is a real and let $T \in L[x]$ be a tree on $\omega \times \lambda$ for some ordinal $\lambda$. Then $\p[T]^V$ is $\bfSigma^1_2$.
\end{lemma}
\begin{proof}
 Let $\gamma$ be an ordinal large enough that $T \in L_\gamma[x]$, for example $\gamma = |\lambda|^{+L[x]}$.
 
 \begin{claim*}
  In $V$, for every real $z$ we have $z \in \p[T]$ if and only if there is an ordinal $\bar{\gamma} < \aleph_1$ and a tree $\bar{T} \in L_{\bar{\gamma}}[x]$ such that $z \in \p[\bar{T}]$ and there is an elementary embedding
  \[ \pi : (L_{\bar{\gamma}}[x] ; \mathord{\in}, x, \bar{T}) \to (L_\gamma[x] ; \mathord{\in}, x, T).\]
  (The language of these structures has a binary predicate symbol for set membership, a unary predicate symbol for $x \subset V_\omega$, and a constant symbol for $\bar{T}$ and $T$.)
 \end{claim*}

 To prove the claim, let $z$ be a real and assume $z \in \p[T]$.  Then $(z,f) \in [T]$ for some $f \in \lambda^\omega$. Let $X$ be the definable closure of $\ran(f)$ in the structure $(L_\gamma[x] ; \mathord{\in}, x, T)$. Then $X$ is the definable closure of $\ran(f) \cup \{T\}$ in the structure $(L_\gamma[x];\mathord{\in},x)$, which has definable Skolem functions, so $X$ is an elementary substructure of $(L_\gamma[x];\mathord{\in},x)$ and by G\"{o}del's condensation lemma for levels of $L[x]$ there is an ordinal $\bar{\gamma}$ and an elementary embedding 
 \[\pi : (L_{\bar{\gamma}}[x] ;\mathord{\in}, x) \to (L_\gamma[x];\mathord{\in},x)\]
 such that $\ran(\pi) = X$.  Because $\ran(f)$ is countable, $X$ is countable, so $\bar{\gamma} < \aleph_1$. Because $T \in X = \ran(\pi)$ we may define a tree $\bar{T} = \pi^{-1}(T)$. Then $\bar{T} \in L_{\bar{\gamma}}[x]$ and we may consider $\pi$ as an elementary embedding of the structure $(L_{\bar{\gamma}}[x] ; \mathord{\in}, x, \bar{T})$ into $(L_\gamma[x] ; \mathord{\in}, x, T)$. Because $\ran(f) \subset X = \ran(\pi)$ we may define an $\omega$-sequence of ordinals $\bar{f} = \pi^{-1} \circ f$. Then we have $(z,\bar{f}) \in [\bar{T}]$ and therefore $z \in \p[\bar{T}]$.
  
 Conversely, let $z$ be a real and assume that there is a countable ordinal $\bar{\gamma}$ and a tree $\bar{T} \in L_{\bar{\gamma}}[x]$ such that $z \in \p[\bar{T}]$ and there is an elementary embedding $\pi$ of $(L_{\bar{\gamma}}[x] ; \mathord{\in}, x, \bar{T})$ into $(L_\gamma[x] ; \mathord{\in}, x, T)$. Because $z \in \p[\bar{T}]$ we have $(z,\bar{f}) \in [\bar{T}]$ for some $\omega$-sequence $\bar{f}$ of ordinals. Then $(z,f) \in [T]$ where $f = \pi \circ \bar{f}$, so $z \in \p[T]$. This completes the proof of the claim.
 
 Note that in the claim, the domain and codomain of $\pi$ are in $L[x]$ because their universes are levels of $L[x]$ but $\pi$ itself is not required to be in  $L[x]$. However, by the absoluteness of elementary embeddability of countable structures the existence of an elementary embedding of $(L_{\bar{\gamma}}[x] ; \mathord{\in}, x, \bar{T})$ into $(L_\gamma[x] ; \mathord{\in}, x, T)$ is absolute between $V$ and $L[x][G]$ where $G \subset \Col(\omega,\bar{\gamma})$ is a $V$-generic filter. (Alternatively, we could take $G$ to be an $L[x]$-generic filter in $V$ because the forcing poset is well-ordered and its power set in $L[x]$ is countable in $V$.)
 
 Let $\kappa = \aleph_1^V$.  In $L[x]$, let $\mathscr{A}$ be the set of all structures in $H_\kappa$ that can be generically elementarily embedded into $(L_\gamma[x] ; \mathord{\in}, x, T)$ and note that because $\mathscr{A}$ is generically hereditary it is ${\bfSigma}_1^{H_\kappa}$ by Proposition \ref{prop:hereditary} and the generic Vop\v{e}nka property of $\kappa$. Because $H_\kappa^{L[x]}$ is equal to $L_\kappa[x]$, which is $\Sigma_1^{\text{HC}}(x)$ in $V$, it follows that $\mathscr{A}$ is ${\bfSigma}_1^{\text{HC}}$ in $V$.
 
 Now in $V$, by the claim and the note on absoluteness following it, for every real $z$ we have $z \in \p[T]$ if and only if there is an ordinal $\bar{\gamma} < \aleph_1$ and a tree $\bar{T} \in L_{\bar{\gamma}}[x]$ such that $z \in \p[\bar{T}]$ and $(L_{\bar{\gamma}}[x] ; \mathord{\in}, x, \bar{T}) \in \mathscr{A}$. Using this condition for membership it follows that $\p[T]$ is also ${\bfSigma}_1^{\text{HC}}$, and because $\p[T]$ is a set of reals this means it is~${\bfSigma}^1_2$.
\end{proof}

\begin{lemma}[ZF] \label{lem:omega-1-is-gen-vop-in-IM}
 Assume $2^{\aleph_1} \not\le |S_{\aleph_1}|$ where $S_{\aleph_1}$ is the pointclass of all $\aleph_1$-Suslin sets.  Then $\aleph_1^V$ is a generic Vop\v{e}nka-like cardinal in every inner model of ZFC.
\end{lemma}
\begin{proof}
 First we will show that $\aleph_1$ has a Vop\v{e}nka-like property in $V$:
 
 \begin{claim*}
  For every countable first-order language $\mathcal{L}$ and every sequence  of $\mathcal{L}$-structures $(\mathcal{M}_\alpha : \alpha < \aleph_1)$ such that the universe of each structure $\mathcal{M}_\alpha$ is a countable ordinal $\mu_\alpha$, there are distinct ordinals $\alpha, \beta < \aleph_1$ and an elementary embedding of $\mathcal{M}_\alpha$ into $\mathcal{M}_\beta$.
 \end{claim*}

 To prove the claim, we first note that each structure $\mathcal{M}_\alpha$ comes with a well-ordering because its universe is an ordinal.  By adding a binary predicate symbol for this well-ordering to the language, we may assume that every structure $\mathcal{M}_\alpha$ has definable Skolem functions. Now let 
 \[\mathcal{L}' = \mathcal{L} \cup \{c_n : n < \omega\}\]
 where each $c_n$ is a new constant symbol. For every ordinal $\alpha < \aleph_1$ and every function $f \in \mu_\alpha^\omega$ let $\mathcal{M}_\alpha^f$ be the expansion of $\mathcal{M}_\alpha$ to $\mathcal{L}'$ defined by \[c_n^{\mathcal{M}_\alpha^f} = f(n)\]
 for all $n<\omega$. Fixing an enumeration of $\mathcal{L}'$ in order type $\omega$, the theory $\operatorname{Th}(\mathcal{M}_\alpha^f)$ is coded by a real that we will call $\operatorname{Code}(\operatorname{Th}(\mathcal{M}_\alpha^f))$. For every set of ordinals $A \subset \aleph_1$ we can define a corresponding set of reals by
 \[A^* = \{ \operatorname{Code}(\operatorname{Th}(\mathcal{M}_\alpha^f)) : \alpha \in A \text{ and } f \in \mu_\alpha^\omega\}.\]
 Note that $A^*$ is $\aleph_1$-Suslin because it is the first coordinate projection of the set
 \[ \{(\operatorname{Code}(\operatorname{Th}(\mathcal{M}_\alpha^f)), \alpha ^\frown f) : \alpha \in S \text{ and } f \in \mu_\alpha^\omega\},\]
 which is closed in $\omega^\omega \times \aleph_1^\omega$ because every $\mathcal{L}'$-formula only contains finitely many of the new constant symbols $c_n$. Then because $2^{\aleph_1} \not\le |S_{\aleph_1}|$ there are two distinct sets $A_0,A_1 \subset \aleph_1$ such that $A_0^* = A_1^*$. We may assume without loss of generality that there is an ordinal $\alpha \in A_0 \setminus A_1$.
 
 Letting $f \in \mu_\alpha^\omega$ be a surjection of $\omega$ onto the countable ordinal $\mu_\alpha$, the real $\operatorname{Code}(\operatorname{Th}(\mathcal{M}_\alpha^f))$ is in $A_0^*$ by definition and is therefore also in $A_1^*$ because $A_0^* = A_1^*$. The membership of $\operatorname{Code}(\operatorname{Th}(\mathcal{M}_\alpha^f))$ in  $A_1^*$ must be witnessed by some ordinal $\beta \in A_1$ and some function $g \in \mu_\beta^\omega$ such that $\operatorname{Code}(\operatorname{Th}(\mathcal{M}_\alpha^f)) = \operatorname{Code}(\operatorname{Th}(\mathcal{M}_\beta^g))$ and therefore 
 \[\operatorname{Th}(\mathcal{M}_\alpha^f) = \operatorname{Th}(\mathcal{M}_\beta^g).\]
 Note that $\alpha \ne \beta$ because $\beta$ is in $A_1$ and $\alpha$ is not. 
 
 Let $\mathcal{F}$ and $\mathcal{G}$ be the definable closures of the sets $\ran(f)$ and $\ran(g)$ in the structures $\mathcal{M}_\alpha$ and $\mathcal{M}_\beta$ respectively. Because these structures have definable Skolem functions, $\mathcal{F}$ and $\mathcal{G}$ are elementary substructures of $\mathcal{M}_\alpha$ and $\mathcal{M}_\beta$ respectively, and because $\operatorname{Th}(\mathcal{M}_\alpha^f) = \operatorname{Th}(\mathcal{M}_\beta^g)$ there is an isomorphism $\pi : \mathcal{F} \to \mathcal{G}$.  (In fact, there is a unique isomorphism $\pi : \mathcal{F} \to \mathcal{G}$ such that $\pi \circ f = g$.) We have $\mathcal{F} = \mathcal{M}_\alpha$ because we chose $f$ to be a surjection onto $\mu_\alpha$, so $\pi$ is an elementary embedding of $\mathcal{M}_\alpha$ into $\mathcal{M}_\beta$, completing the proof of the claim.
 
 Continuing with the proof of the lemma we let $\kappa = \aleph_1^V$, let $W$ be an inner model of ZFC, and let $\mathscr{B} \in W$ be a set of structures for a common first-order language $\mathcal{L}$ such that $|\mathcal{L}|^W < \kappa$, $|\mathscr{B}|^W = \kappa$, and $|\mathcal{M}|^W < \kappa$ for every $\mathcal{M} \in \mathscr{B}$. We may enumerate $\mathscr{B}$ in $W$ as a $\kappa$-sequence of distinct structures  $(\mathcal{M}_\alpha : \alpha < \kappa)$. Using the axiom of choice in $W$ to choose a well-ordering of each structure, we may assume without loss of generality that the universe of each $\mathcal{M}_\alpha$ is an ordinal $\mu_\alpha < \kappa$. 
 
 Now by the claim there are distinct ordinals $\alpha, \beta < \kappa$ and an elementary embedding of $\mathcal{M}_\alpha$ into $\mathcal{M}_\beta$ in $V$. By the absoluteness of elementary embeddability of countable structures it follows that there is an elementary embedding of $\mathcal{M}_\alpha$ into $\mathcal{M}_\beta$ in $W[G]$ where $G \subset \Col(\omega, \mu_\alpha)$ is a $V$-generic filter, so the statement ``there is a generic elementary embedding of $\mathcal{M}_\alpha $ into $\mathcal{M}_\beta$'' holds in $W$.
\end{proof}

The main theorem will now follow easily from Lemmas \ref{lem:downward-absolute}, \ref{lem:small-forcing}, \ref{lem:suslin-is-sigma-1-2}, and \ref{lem:omega-1-is-gen-vop-in-IM}:

\begin{proof}[Proof of Theorem \ref{thm:main-equicon}]\ \\

 $\Con\eqref{item:gen-vopenka}$ implies $\Con\eqref{item:aleph-1-suslin-is-sigma-1-2}$: Assume ZFC and let $\kappa$ be a generic Vop\v{e}nka cardinal. Then $\kappa$ is a generic Vop\v{e}nka cardinal in $L$ by Lemma \ref{lem:downward-absolute}. Letting $G \subset \Col(\omega,\mathord{<}\kappa)$ be an $L$-generic filter we have $\aleph_1^{L[G]} = \kappa$, $\aleph_2^{L[G]} = \kappa^{+L}$, and $L[G] \models \text{CH}$ by the inaccessibility of $\kappa$ in $L$ and well-known properties of the Levy collapse. We will show that theory \ref{item:aleph-1-suslin-is-sigma-1-2} holds in $L(\mathbb{R})^{L[G]}$.

 First we will show that $L(\mathbb{R})^{L[G]}$ satisfies the statement $\Theta = \aleph_2$, meaning that there is no surjection from the reals onto $\aleph_2$. This statement holds in $L[G]$ because it follows from ZFC + CH. Note that the models $L[G]$ and $L(\mathbb{R})^{L[G]}$ have the same $\aleph_1$, namely $\kappa$, and the same $\aleph_2$, namely $\kappa^{+L}$. Because they also have the same reals, the statement $\Theta = \aleph_2$ is downward absolute from $L[G]$ to $L(\mathbb{R})^{L[G]}$.
 
 Now let $A$ be an $\aleph_1$-Suslin set of reals in $L(\mathbb{R})^{L[G]}$.\footnote{The following argument would apply more generally to any Suslin set $A$, but this would not yield a more general result because $\Theta = \aleph_2$ implies that every Suslin set is $\aleph_1$-Suslin.} We will show that $A$ is ${\bfSigma}^1_2$ in $L(\mathbb{R})^{L[G]}$, or equivalently in $L[G]$. Take a tree $T$ on $\omega \times \kappa$ in $L(\mathbb{R})^{L[G]}$ such that $A = \p[T]$. Because $T$ is in $L(\mathbb{R})^{L[G]}$ it follows that $T \in \text{HOD}_{z}^{L[G]}$ for some real $z \in L[G]$. Take a $\Col(\omega,\mathord{<}\kappa)$-name $\dot{z} \in L$ such that $z = \dot{z}_G$.
 
 For every ordinal $\alpha < \kappa$ we can define an $L$-generic filter on $\Col(\omega,\mathord{<}\alpha)$ by 
 \[G_\alpha = G \cap \Col(\omega,\mathord{<}\alpha).\]
 Take a successor ordinal $\alpha < \kappa$ sufficiently large that for every $n < \omega$ the value of $\dot{z}(n)$ is decided by a condition in $G_\alpha$.  (This is possible because $\kappa$ has uncountable cofinality in $L[G]$.)  Then we have $z \in L[G_\alpha]$ and a standard homogeneity argument shows that $\text{HOD}_{z}^{L[G]} \subset L[G_\alpha]$, so $T \in L[G_\alpha]$. Because $\alpha$ is a successor ordinal it is collapsed by forcing with $\Col(\omega,\mathord{<}\alpha)$, so the generic filter $G_\alpha$ is countable in $L[G_\alpha]$. Taking a real $x \in L[G_\alpha]$ coding $G_\alpha$, we have $L[x] = L[G_\alpha]$ and therefore $T \in L[x]$.
 
 Because $L[x]$ is a generic extension of $L$ by the poset $\Col(\omega,\mathord{<}\alpha) \in V_\kappa^L$, the fact that $\kappa$ is a generic Vop\v{e}nka cardinal in $L$ implies that $\kappa$ is a generic Vop\v{e}nka cardinal in $L[x]$ by Lemma \ref{lem:small-forcing}. Then because $T \in L[x]$ and the cardinal $\aleph_1^{L[G]} = \kappa$ is a generic Vop\v{e}nka cardinal in $L[x]$, the set of reals $A = \p[T]$ is ${\bfSigma}^1_2$ in $L[G]$ by Lemma \ref{lem:suslin-is-sigma-1-2}.
 
 $\eqref{item:aleph-1-suslin-is-sigma-1-2}$ implies $\eqref{item:no-injection-into-S-aleph-1}$: Assume toward a contradiction that ZF holds, $S_{\aleph_1} = {\bfSigma}^1_2$, $\Theta = \aleph_2$, and $2^{\aleph_1} \le S_{\aleph_1}$. Note that $\aleph_2 \le^* 2^{\aleph_1}$ and ${\bfSigma}^1_2 \le^* 2^{\aleph_0}$ provably in ZF, and  $2^{\aleph_1} \le S_{\aleph_1}$ implies $2^{\aleph_1} \le^* S_{\aleph_1}$, where $\le^*$ denotes the existence of a surjection, so we have 
 \[\aleph_2 \le^* 2^{\aleph_1} \le^* S_{\aleph_1} = {\bfSigma}^1_2 \le^* 2^{\aleph_0},\]
 contradicting $\Theta = \aleph_2$.
  
 $\Con\eqref{item:no-injection-into-S-aleph-1}$ implies $\Con\eqref{item:gen-vopenka}$: Assume ZF and $2^{\aleph_1} \not\leq |S_{\aleph_1}|$.  Then $\aleph_1^V$ is a generic Vop\v{e}nka-like cardinal in $L$ by Lemma \ref{lem:omega-1-is-gen-vop-in-IM}, and the least generic Vop\v{e}nka-like cardinal in $L$ is a generic Vop\v{e}nka cardinal in $L$ by Propositions \ref{prop:singular-equivalence} and \ref{prop:omega-erdos-is-gen-vop-lim-of-gen-vop}.  This completes the proof of Theorem~\ref{thm:main-equicon}.
\end{proof}

% \bibliographystyle{plain}
% \bibliography{math}

\end{document}